\documentclass[11pt,a4paper]{amsart}
\usepackage{amsmath,amssymb,mathtools,booktabs,array}
\usepackage[margin=27mm]{geometry}
\usepackage{xcolor}
\usepackage[unicode,colorlinks=true,linkcolor=blue!45!black,citecolor=blue!45!black,urlcolor=blue!45!black]{hyperref}
\usepackage[nameinlink]{cleveref}
\usepackage[
    backend=biber,
    style=numeric-comp,
    giveninits=true,
    sorting=nyt,doi=false,url=false,isbn=false
]{biblatex}
\DeclareFieldFormat[article,inbook,incollection,inproceedings,unpublished]
  {title}{\mkbibemph{#1}}
\renewbibmacro*{in:}{%
  \ifentrytype{article}
    {}
    {\bibstring{in}\printunit{\intitlepunct}}}
\DeclareFieldFormat{journaltitle}{#1\isdot}

\newtheorem{theorem}{Theorem}[section]
\newtheorem{lemma}[theorem]{Lemma}
\newtheorem{proposition}[theorem]{Proposition}

\newtheorem{definition}[theorem]{Definition}
\theoremstyle{remark}\newtheorem{remark}[theorem]{Remark}
\newcommand{\R}{\mathbb R}
\newcommand{\Z}{\mathbb Z}
\newcommand{\T}{\mathbb T}
\newcommand{\N}{\mathbb N}

\newcommand{\dd}{\,\mathrm d}
\newcommand{\supp}{\operatorname{supp}}

\newcommand{\Hess}{\operatorname{Hess}}
\newcommand{\diag}{\operatorname{diag}}
\newcommand{\sgn}{\operatorname{sgn}}
\newcommand{\norm}[1]{\left\lVert#1\right\rVert}

\newcommand{\eps}{\varepsilon}
\allowdisplaybreaks[2]

\title[Sharp decay estimates for the DFS]{Sharp Decay estimates for the discrete Fourth-Order Schr\"odinger Equation}

\author{Jiawei Cheng}
\address{Jiawei Cheng: Academy of Mathematics and Systems Sciences, Chinese Academy of Sciences}
\email{\href{mailto:jiawei.cheng.1999@gmail.com}{jiawei.cheng.1999@gmail.com};{jiawei.cheng@amss.ac.cn}}

\date{}
\hypersetup{pdftitle={Sharp decay for DFS in arbitrary dimensions},pdfauthor={Research draft}}

\begin{document}

\begin{abstract}
We establish sharp time-decay estimates for the discrete fourth-order Schr\"odinger equation on lattices of arbitrary dimension, extending the author's previous work \cite{C24}. In higher dimensions, scalar Fourier factorization reduces the
analysis to product estimates for one-dimensional oscillatory integrals. Since this method does not yield the sharp decay exponents in low dimensions, we also use Newton polyhedra to obtain the required uniform estimates. In particular, the sharp decay rate for the biharmonic Schr\"odinger propagator is $|t|^{-d/4}$.
\end{abstract}

\maketitle
% \tableofcontents
% \clearpage

\section{Introduction}
The fourth-order Schr\"{o}dinger equation on Euclidean space takes the form
\begin{equation}\label{equ-continuous fourth-order}
    i\partial_t u + \Delta^2 u + \varepsilon \Delta u  = F(u).
\end{equation}
%where $u:\R \times \R^d \rightarrow \mathbb C$ is a complex-valued function and $\varepsilon \in \{-1,0,1\}$. 
With a power-type nonlinearity, this model was introduced by Karpman \cite{K96} and Shagalov \cite{KS00} to take into account the role of small fourth-order dispersion terms in the propagation of intense laser beams in a bulk medium with Kerr nonlinearity. Since then, \eqref{equ-continuous fourth-order} has been extensively studied, including its dispersive estimates, deterministic and probabilistic well-posedness, scattering, and wave operators. See Ben-Artzi et al. \cite{BKS00}, Pausader \cite{P07,P09}, Miao et al. \cite{MXZ09,MXZ11}, \cite{D21,GW02,HHW06,GG21,YYZ24} and the references therein. 

Analysis and partial differential equations on discrete objects have attracted wide attention in connection with applications such as image processing and neural networks, see \cite{B17,D93,FH10,G18,K11}. Discrete dispersive equations also provide natural numerical models for physical phenomena. On general graphs, see \cite{FT04,HH24,LX19,LX22} for related work. 

The standard lattice graph,
\[
\Z^d = \{(x_1,\cdots,x_d) \in \R^d: x_j \in \Z, j=1,\cdots,d\},
\]
is one of the most classical and important discrete objects. The discrete Laplacian on it is given by 
\[
\Delta f(x)=\sum_{j=1}^d\bigl(f(x+e_j)+f(x-e_j)-2f(x)\bigr),\qquad x\in\Z^d.
\]

On $\Z^d$, it is straightforward to introduce the Fourier transform and its inverse (defined on $\T^d=[-\pi,\pi]^d$), together with the associated identities and inequalities. Thus, under the discrete Fourier transform, the operator $\Delta$ has the symbol shown below:
\[
\mathcal{F}(-\Delta f)(\xi) = \sum_{j=1}^d(2-2\cos\xi_j)\mathcal{F}(f)(\xi):=\omega(\xi)^2 \mathcal{F}(f)(\xi).
\]
While dealing with dispersive equations, one typically begins with the decay estimate for the $l^{\infty}$ norm of the solution in terms of time and the $l^1$ norm of the initial data. This reduces to obtaining pointwise estimates for the corresponding fundamental solution.

For the second-order Schr\"odinger model, Stefanov and Kevrekidis \cite{SK05} observed a special separation-of-variables property from the explicit formula of $\omega(\xi)$ and studied the decay problem completely. For other equations, however, this separation property is unavailable and one must analyze oscillatory integrals in several variables. Nevertheless, Bi, Cheng and Hua \cite{BCH26,BCH23} established uniform estimates using Newton polyhedra in the study of the wave equation, extending the well-known result of Schultz \cite{S98}. Their strategy is quite effective when the number of degenerate variables is no more than four. See \cite{BG17,CQ24,C25,CA23,GLY25,HTY25,HW24,HY19,W25,W26,WZ25,WZ25-2,WZ25-3} for other models, for which the analysis of high-dimensional integrals remains challenging.

Recently in \cite{YZ26}, You and Zhan made an insightful observation that the fundamental solution for the discrete wave equation admits a proper decomposition, reducing the higher-dimensional analysis to uniform estimates for one-dimensional oscillatory integrals. Similarly in \cite{C26}, Chen used an auxiliary scalar Fourier variable to decouple a multidimensional oscillatory integral.

In this article, we consider the discrete fourth-order Schr\"{o}dinger equation (DFS, in short) with Cauchy data
\begin{equation}\label{equ-DFS}
    \left\{
    \begin{aligned}
        & i \partial_t u(x,t) + {\Delta}^2 u(x,t) - \gamma \Delta u(x,t) = F(u(x,t)), \\
        & u(x,0) = f(x),
    \end{aligned}
    \right.
\end{equation}
where $x \in \mathbb{Z}^d$, $t\in \R$ and the parameter $\gamma \in \R$. 
In the following, we pursue sharp decay estimates for the fundamental solution of DFS, i.e. (see \eqref{equ-semigroup solution})
\[
G_{d,\gamma}(x,t) =(2\pi)^{-d}\int_{\T^d}e^{ix\cdot\xi+it(\omega(\xi)^4+\gamma\omega(\xi)^2)}\dd\xi.
\]
We determine the sharp decay exponents for $G_{d,\gamma}$ in arbitrary dimensions by combining the scalar Fourier
factorization (for $d\ge 4$) and Newton polyhedra (for $1\le d \le 3$). It is expected that the decay rate for $G_{d,\gamma}$ depends on both the dimension and the parameter. In general, we partition the parameter space $\R$ into three sets: 
\[
\mathcal E_d=(-\infty,-8d)\cup(0,\infty),\qquad
\mathcal I_d=(-8d,0),\qquad \mathcal B_d=\{-8d,0\}.
\]
Now the main conclusion is as follows.
\begin{theorem}\label{thm-main}
%[All dimensions and all real parameters]
Define \((\beta,p)=(\beta_{d,\gamma},p_{d,\gamma})\) by the following table.
\begin{center}\small
\begin{tabular}{@{}lll@{}}\toprule
Dimension & Parameter & \((\beta,p)\)\\\midrule
\(d=1\)&\(\gamma=0,-8\)&\((1/4,0)\)\\
&otherwise &\((1/3,0)\)\\\addlinespace
\(d=2\)&\(\gamma\in\mathcal E_2\)&\((3/4,0)\)\\
&\(\gamma=-8\)&\((1/2,1)\)\\
&\(\gamma\in \mathcal I_2 \cup \mathcal B_2 \setminus\{-8\}\)&\((1/2,0)\)\\\addlinespace
\(d=3\)&\(\gamma\in\mathcal E_3\)&\((7/6,0)\)\\
&\(\gamma\in\mathcal B_3\)&\((3/4,0)\)\\
&\(\gamma\in\mathcal I_3\)&\((1/2,0)\)\\\addlinespace
\(d=4\)&\(\gamma\in\mathcal E_4\)&\((3/2,1)\)\\
&\(\gamma\in\mathcal B_4\)&\((1,0)\)\\
&\(\gamma\in\mathcal I_4\)&\((1/2,0)\)\\\addlinespace
\(d\ge5\)&\(\gamma\in\mathcal E_d\)&\(((2d+1)/6,0)\)\\
&\(\gamma\in\mathcal B_d\)&\((d/4,0)\)\\
&\(\gamma\in\mathcal I_d\)&\((1/2,0)\)\\\bottomrule
\end{tabular}\end{center}
There exist \(C,c,T>0\) such that
\begin{align*}
\sup_{x\in\Z^d}|G_{d,\gamma}(x,t)|
&\le C(1+|t|)^{-\beta}\log^p(2+|t|),&&t\in\R.
\end{align*}
We point out that both the powers and the logarithmic factors in the table are sharp. Constants may depend on the fixed dimension and parameter, while they are uniform in $x \in \Z^d$.
\end{theorem}

%\begin{corollary}[The pure fourth-order model]
%For every \(d\ge1\),
%\[
%\norm{e^{it\Delta^2}}_{\ell^1(\Z^d)\to\ell^\infty(\Z^d)}
%\asymp_d |t|^{-d/4}\qquad(|t|\ge T_d).
%\]
%\end{corollary}
\begin{remark}
Theorem \ref{thm-main} is sharp in the sense that
\[
\sup_{x\in\Z^d}|G_{d,\gamma}(x,t)| \ge c_{d,\gamma}|t|^{-\beta}\log^p|t|, \qquad |t|\ge T.
\]
\end{remark}

\begin{remark}
    When $d=1,2$, the author has proved such estimates for DFS with all parameters. In the proof, he established uniform estimates for the related oscillatory integral under analytic perturbations of the phase, which is a stronger conclusion than Theorem \ref{thm-main} since the perturbation in $G_{d,\gamma}$ is linear. Meanwhile, this theorem also provides a correction to Theorems 1.1 and 3.8 of \cite{C24}. 
\end{remark}

\begin{remark}
The approach of You-Zhan and the method of Bi-Cheng-Hua are perfectly complementary to each other. For an oscillatory integral, the index given by scalar Fourier factorization is sharp only when the dimension is high, while the uniform estimate by Newton polyhedra is always available when there are relatively few variables.
\end{remark}

The paper is organized as follows. We exhibit the detailed process of scalar Fourier factorization and recall a few basic estimates in Section \ref{sec-factor}. Section \ref{sec-high} and \ref{sec-low} apply two different methods to high ($d\ge 4$) and low ($1 \le d \le 3$) dimensions, respectively. Finally, we prove the sharpness of the main theorem in Section \ref{sec-sharp}. The \hyperref[sec-app]{Appendix} collects facts on the discrete setting and uniform estimates.

%Following the notation of the supplied DFS manuscript \cite{C24}, set
%\begin{align}
%\omega(\xi)^2&=\sum_{j=1}^d(2-2\cos\xi_j),&
%\phi_\gamma(\xi)&=\omega(\xi)^4+\gamma\omega(\xi)^2,\label{eq-phi}\\
%G_{d,\gamma}(x,t)&=(2\pi)^{-d}\int_{\T^d}e^{ix\cdot\xi+it\phi_\gamma(\xi)}\dd\xi,&
%\varphi(v,\xi)&=v\cdot\xi+\phi_\gamma(\xi).\label{eq-kernel}
%\end{align}
%Then \(u(t)=G_{d,\gamma}(t)*u(0)\). 
%Constants may depend on the fixed dimension and parameter. No uniformity across the parameter boundaries is asserted.

\section{Scalar Fourier Factorization}\label{sec-factor}
$G_{d,\gamma}$ is obtained by applying the discrete Fourier transform to \eqref{equ-DFS} with $F\equiv 0$, solving the ODE and taking the inverse transform. See Appendix \ref{ssec-disc set} for the definitions and the detailed calculation. We then write $G_{d,\gamma}$ in the following way.

Set
\[
S(\xi)=-\sum_{j=1}^d\cos\xi_j \in [-d,d],\qquad
q_\gamma(s)=(2d+2s)^2+\gamma(2d+2s).
%~~~~~\phi_\gamma(\xi) =\omega(\xi)^4+\gamma\omega(\xi)^2.
\]
Note that $q_\gamma\circ S$ indeed recovers the dispersive relation of \eqref{equ-DFS}. We also fix a cutoff \(a\in C_c^\infty(\R)\), such that \(a=1\) on \([-d,d]\). Hence we have
\[
G_{d,\gamma}(x,t)=(2\pi)^{-d}\int_{\T^d}e^{ix\cdot\xi+itq_\gamma(S(\xi))}a(S(\xi))\dd\xi.
\]
Then we observe that $S$ has the separation-of-variables form, so we identify the term containing $S$, i.e.
\[
F_t(s)=a(s)e^{itq_\gamma(s)},
\]
as the Fourier inverse of its Fourier transform, which is
\[
F_t(S(\xi)) = (2\pi)^{-1}\int_{\R}e^{i\eta S(\xi)}\widehat F_t(\eta)\dd \eta.
\]
Define
\[
B_m(\eta)=\int_{-\pi}^{\pi}e^{im\theta-i\eta\cos\theta}\dd\theta.
\]
Then Fubini's theorem gives
\[
\begin{aligned}
    G_{d,\gamma}(x,t)
    & =(2\pi)^{-d}\int_{\T^d}e^{ix\cdot\xi}F_t(S(\xi))\dd\xi\\
    & =\frac1{(2\pi)^{d+1}}\int_\R\widehat F_t(\eta)
    \dd \eta \int_{\T^d}e^{i(x\cdot \xi+\eta S(\xi))}\dd \xi\\
    & =\frac1{(2\pi)^{d+1}}\int_\R\widehat F_t(\eta)\prod_{j=1}^dB_{x_j}(\eta)\dd\eta.
\end{aligned}
\]

We now estimate each part in the last formula. Keep in mind that after the factorization, we only need to deal with integrals with one variable, making the analysis much easier.

The first lemma concerns the scalar Fourier estimate.
\begin{lemma}\label{lem-Ft}
For fixed \(a\) and \(|t|\ge1\),
    \begin{equation}\label{eq-Ftbound}
    \sup_{\eta\in\R}|\widehat F_t(\eta)|\le C_a|t|^{-1/2}.
    \end{equation}
If \(|q_\gamma'|\ge c_0>0\) on \(\supp a\), there are \(0<c<C\) such that, for every \(N\),
    \begin{equation}\label{eq-tail}
    \int_{\{|\eta|\notin[c|t|,C|t|]\}}|\widehat F_t(\eta)|\dd\eta
    \le C_{a,N}|t|^{-N}.
    \end{equation}
\end{lemma}

\begin{proof}
The scalar phase is \(tq_\gamma(s)-\eta s\), whose second derivative is \(8t\). The second-derivative van der Corput lemma (Theorem \ref{lemma-van der corput}) gives \eqref{eq-Ftbound}.
If \(q_\gamma'\) is also bounded away from zero on the fixed support, an enlarged constant-ratio annulus has the property that
\(|tq_\gamma'(s)-\eta|\gtrsim|t|+|\eta|\) outside it.
Repeated integration by parts yields
\(|\widehat F_t(\eta)|\le C_L(|t|+|\eta|)^{-L}\).
Integration in \(\eta\), with \(L\) sufficiently large, proves \eqref{eq-tail}.
\end{proof}

Then we deal with $B_m$ and the product integrability. 
The next two estimates are Lemmas 4.1 and 4.2 of \cite{YZ26}, so we omit their proofs.
\begin{lemma}\label{lem-B}
Let \(\rho(u)=(1+|u|)^{-1/4}\). For all \(m\in\Z\) and \(|\eta|\ge1\),
    \begin{equation}\label{eq-B}
        |B_m(\eta)|\le C\left[|\eta|^{-1/2}
        +|\eta|^{-1/3}\sum_{\sigma=\pm1}
        \rho\left(\frac{m+\sigma\eta}{|\eta|^{1/3}}\right)\right].
    \end{equation}
\end{lemma}
For the specific function $\rho$, we have 
\begin{lemma}\label{lem-rho}
    %[Products of translated envelopes]
If an interval \(J\) has length at most \(L\ge1\), then, for arbitrary \(a_j\in\R\) and \(\sigma_j=\pm1\),
\[
\int_J\prod_{j=1}^k\rho(\sigma_ju+a_j)\dd u\lesssim_k
\begin{cases}
L^{1-k/4},&1\le k<4,\\
\log(2+L),&k=4,\\
1,&k>4.
\end{cases}
\]
\end{lemma}
To close this section, we prove the following dyadic product bounds.
\begin{proposition}\label{prop-annulus}
For \(R\ge1\) and arbitrary \(x\in\Z^d\),
\begin{equation}\label{eq-annulus}
    \int_{R\le|\eta|\le2R}\prod_{j=1}^d|B_{x_j}(\eta)|\dd\eta\lesssim_d H_d(R),
    \qquad
    H_d(R)=\begin{cases}
    R^{1-d/2},&d=1,2,3,\\
    R^{-1}\log(2+R),&d=4,\\
    R^{-(d-1)/3},&d\ge5.
    \end{cases}
\end{equation}
In particular, for \(d\ge3\),
\begin{equation}\label{eq-globalproduct}
    \sup_{x\in\Z^d}\int_\R\prod_{j=1}^d|B_{x_j}(\eta)|\dd\eta<\infty.
\end{equation}
\end{proposition}
\begin{proof}
Expand the product of the bounds in Lemma \ref{lem-B}. 
Applying Lemma \ref{lem-rho} shows that, after the substitution, \(\eta=R^{1/3}u\), each term containing \(k\) turning-point factors and \(d-k\) regular factors is bounded by 
%a constant times
\[
R^{-(d-k)/2-k/3+1/3}\int_{J_R}\prod_{\ell=1}^k\rho(\sigma_\ell u+a_\ell)\dd u,
\qquad |J_R|\lesssim R^{2/3}.
\]
For \(k<4\), this is \(O(R^{1-d/2})\), including \(k=0\) by direct integration.
For \(k=4\), it is \(O(R^{1-d/2}\log(2+R))\).
For \(k>4\), it is \(O(R^{-d/2+k/6+1/3})\), with the slowest exponent \(-(d-1)/3\) at \(k=d\).
When \(d\ge5\), the first two types decay faster. This proves \eqref{eq-annulus}.
Use the trivial bound on \(|\eta|\le1\), and sum over \(R=2^j\). The sum converges for every \(d\ge3\), uniformly in all translation parameters \(x_j\), proving \eqref{eq-globalproduct}.
\end{proof}

\section{High-dimensional upper bounds}\label{sec-high}

\subsection{Exterior and interior parameter intervals}
Recall that we divide the domain of parameters into three parts:
\[
\mathcal E_d=(-\infty,-8d)\cup(0,\infty),\qquad
\mathcal I_d=(-8d,0),\qquad \mathcal B_d=\{-8d,0\}.
\]
By direct calculation, we get
\begin{equation}\label{eq-qder}
    q_\gamma'(s)=8d+8s+2\gamma,\qquad q_\gamma''(s)=8.
\end{equation}
Therefore for the first part, one can choose \(a=1\) near \([-d,d]\) while satisfying the second hypothesis in Lemma \ref{lem-Ft}.
\begin{proposition}\label{prop-external}
For \(\gamma\in\mathcal E_d\) and \(|t|\ge1\),
\[
\sup_x|G_{d,\gamma}(x,t)|\lesssim
\begin{cases}
|t|^{-(d-1)/2},&d=2,3,\\
|t|^{-3/2}\log(2+|t|),&d=4,\\
|t|^{-(2d+1)/6},&d\ge5.
\end{cases}
\]
%The same conclusion holds for a fixed scalar cutoff kernel \(G_a\) whenever \(|q_\gamma'|\ge c>0\) on its scalar support.
\end{proposition}
\begin{proof}
Equations \eqref{eq-Ftbound} and \eqref{eq-tail} give
\[
|G_{d,\gamma}(x,t)|\lesssim |t|^{-1/2}\int_{c|t|\le|\eta|\le C|t|}\prod_j|B_{x_j}(\eta)|\dd\eta
+O_N(|t|^{-N}).
\]
Cover the constant-ratio annulus by finitely many dyadic annuli and apply \eqref{eq-annulus}.
\end{proof}

%\begin{remark}
%    This proposition gives a nonsharp upper bound in low dimensions. We will show the needed improvements in Section \ref{sec-low}.
%\end{remark}

For \(\gamma\in\mathcal I_d\), the derivative \(q_\gamma'\) vanishes inside \((-d,d)\), see \eqref{eq-qder}. The contribution near \(\eta=0\) cannot be discarded.

\begin{proposition}\label{prop-interior}
For \(d\ge3\) and every fixed \(\gamma\in\R\),
\[
\sup_x|G_{d,\gamma}(x,t)|\lesssim_{d,\gamma}(1+|t|)^{-1/2}.
\]
For \(\gamma\in\mathcal I_d\), this is the upper bound required in Theorem \ref{thm-main}.
\end{proposition}
\begin{proof}
This follows by substituting \eqref{eq-Ftbound} and \eqref{eq-globalproduct} into the representation of $G_{d,\gamma}$.
\end{proof}

\subsection{Endpoint parameters}\label{sec-endpoint}
We first take \(\gamma=0\). Near the origin,
\[
\omega(\xi)^2=|\xi|^2+O(|\xi|^4),\qquad
\phi_0(\xi)=|\xi|^4+O(|\xi|^6),
\]
and hence
\begin{equation*}
\Hess\phi_0(\xi)=8\xi\xi^{\mathsf T}+4|\xi|^2I_d+O(|\xi|^4).
\end{equation*}
On a sufficiently small fixed neighborhood, \(\Hess\phi_0(\xi)\ge c|\xi|^2I_d\). We then give the uniform estimate under linear perturbations at a quartic corner directly.

\begin{lemma}\label{lem-quartic}
Let \(f(x)=|x|^4+O(|x|^6)\) be real analytic. For amplitudes supported sufficiently close to the origin,
\[
\sup_{v\in\R^d}\left|\int e^{it(f(x)+v\cdot x)}\psi(x)\dd x\right|
\lesssim(1+|t|)^{-d/4}\norm\psi_{C^N}.
\]
\end{lemma}
\begin{proof}
Write \(J(t,v)\) for the integral. The case \(|t|\le1\)
is immediate, so let \(T=|t|\ge1\).
Choose \(0<\delta\le1\) sufficiently small and assume
\(\operatorname{supp}\psi\subset B(0,\delta)\).
A smooth dyadic partition gives
\[
J(t,v)=\sum_{j\ge0}J_j(t,v),\qquad
J_j(t,v)=\int e^{it(f(x)+v\cdot x)}
\psi(x)\theta(x/r_j)\,dx,
\]
where \(r_j=\delta2^{-j}\),
\(\operatorname{supp}\theta\subset\{1/2\le|x|\le2\}\),
and \(\sum_{j\ge0}\theta(x/r_j)=1\) on
\(0<|x|<\delta\).

Set \(r=r_j\), \(f_r(y)=r^{-4}f(ry)\), and \(b=r^{-3}v\).
After \(x=ry\), the oscillation parameter is \(Tr^4\).
Analyticity gives
\(f_r(y)=|y|^4+O_{C^m}(r^2)\) for each fixed \(m\)
on the fixed annulus. In particular,
\[
D^2f_r(y)=8yy^{\mathsf T}+4|y|^2I_d+O(r^2)
\ge cI_d,
\qquad
\|\psi(r\cdot)\theta\|_{C^N}\lesssim\|\psi\|_{C^N},
\]
uniformly in \(r\).
Since the linear perturbation does not change the Hessian,
uniform nondegenerate stationary phase, combined with
the trivial bound, yields
\[
|J_j(t,v)|
\lesssim \|\psi\|_{C^N}r_j^d
\min\{1,(Tr_j^4)^{-d/2}\}.
\]
The constant is independent of \(v\): bounded \(b\)
is covered by stationary phase with parameters,
while sufficiently large \(|b|\) is treated by
nonstationary integration by parts.

Finally, splitting at \(r_*=T^{-1/4}\) and summing
the two geometric series gives
\[
\begin{aligned}
|J(t,v)|
&\lesssim \|\psi\|_{C^N}
\left(\sum_{r_j\le r_*}r_j^d
+T^{-d/2}\sum_{r_j>r_*}r_j^{-d}\right)\\
&\lesssim \|\psi\|_{C^N}
\left(r_*^d+T^{-d/2}r_*^{-d}\right)
\lesssim T^{-d/4}\|\psi\|_{C^N},
\end{aligned}
\]
uniformly in \(v\), as required.

\end{proof}

\begin{proof}[Proof of Theorem \ref{thm-main} when $\gamma=0$]
    
To specify the complementary scalar cutoff, choose \(a\in C_c^\infty(\R)\) equal to one near \([-d,d]\), and \(\chi_0\in C_c^\infty(\R)\) equal to one for \(|u|\le\eps\) and zero for \(|u|\ge2\eps\). Define
\[
a_{\rm low}(s)=a(s)\chi_0(s+d),\qquad
a_{\rm high}(s)=a(s)(1-\chi_0(s+d)).
\]
Their compositions with \(S\) add to one on the torus, so \(G=G_{a_{\rm low}}+G_{a_{\rm high}}\).
Because \(S(\xi)+d=\omega(\xi)^2/2\), the low cutoff is supported only near the origin. For sufficiently small \(\eps\), Lemma \ref{lem-quartic} gives \(t^{-d/4}\) for that part.
The remaining scalar cutoff is exactly \(a_{\rm high}\). On its support,
\(|q_0'(s)|=8|s+d|\ge8\eps\), so Proposition \ref{prop-external} applies. 
Therefore, both contributions are $\mathcal{O}(t^{-1/2})$ when d = 2. For $d \ge 3$, the high-frequency contribution decays faster, so the low-frequency contribution determines the resulting bound.
Dimension one is handled directly in the next section.
\end{proof}

\begin{remark}
    The other endpoint follows from the exact reflection identity
    \begin{equation}\label{eq-reflection}
    \phi_\gamma(\pi\mathbf1-\xi)=16d^2+4d\gamma+\phi_{-8d-\gamma}(\xi).
    \end{equation}
    Thus the endpoints \(\gamma=-8d\) and \(\gamma=0\) have identical bounds for the absolute values of their kernels, up to the corresponding reflection of the lattice variable.
\end{remark}

\begin{remark}
Lemma \ref{lem-quartic} concerns only linear perturbations. For \(d>2\), the best exponent for \(|x|^4\) under arbitrarily small analytic perturbations is only \(1/2\): the perturbation \(-2\eps|x|^2\) creates a critical sphere.
We do not use the false assertion \(M(|x|^4)\curlyeqprec(-d/4,0)\) for that larger perturbation class.
\end{remark}

\section{Low-dimensional supplements}\label{sec-low}
When $d \ge 4$, we will see in the next section that the decay rates given by Proposition \ref{prop-external} and \ref{prop-interior} are indeed sharp. Nevertheless, we still need to treat the remaining low-dimensional cases, some of which were addressed in \cite{C24}. For completeness, we treat such cases in a unified framework.
  
In this section, we will apply the oscillatory integral theory, mainly from Varchenko \cite{V76} and Karpushkin \cite{K83,K84}. In particular, we focus on the uniform estimates in dimensions one and two, which have been well established. 
Here we only summarize the main definitions and results, see \hyperref[sec-app]{Appendix} for more details.

We write \(M(h)\curlyeqprec(\beta,p)\) for a uniform estimate with exponent pair ($\beta$,$p$), for the integral with phase function $h$ under sufficiently small analytic perturbations. The oscillatory index is always determined by the Newton polyhedron associated with the fixed phase. 
In the two-dimensional setting, adapted
coordinates allow us to update the Newton estimate to the uniform estimate (Theorem \ref{theorem-2dim NP uniform}).

We establish analytic stability for each possible degenerate critical point within the torus $\T^d$. Together with Lemma \ref{lemma-regular and nondegenerate estimate}, these stability results enable us to cover some neighborhood of any velocity. Then compactness gives a finite covering since for \(|v|>2\sup|\nabla\phi_\gamma|+1\), a phase derivative is bounded away from zero, and repeated integration by parts gives arbitrarily fast decay. The full torus integral is patched only when \(tv=x\in\Z^d\); local integrals may use arbitrary real velocities.
The low-dimensional estimates obtained below, together with the preceding section, yield the required upper bound.

In this section we write
\begin{equation*}
    \phi_\gamma(\xi)=\omega(\xi)^4+\gamma\omega(\xi)^2,\qquad
    \varphi(v,\xi)=v\cdot\xi+\phi_\gamma(\xi),
\end{equation*}
and 
\begin{equation*}
A=2\omega^2+\gamma,\quad s_j=\sin\xi_j,\quad c_j=\cos\xi_j,
\quad\Hess\phi_\gamma=2A\diag(c_j)+8ss^{\mathsf T}.
\end{equation*}

\subsection{Dimension one}
Put \(c=\cos\xi\), \(s=\sin\xi\), and \(B=4+\gamma\). Then
\[
\phi_\gamma''=2Bc+8-16c^2,\qquad
\phi_\gamma'''=2s(16c-B).
\]
If \(s\ne0\) and the third derivative vanishes, then \(c=B/16\), and
\(\phi_\gamma''=8+B^2/16>0\).
If \(s=0\), the second derivative can vanish only at
\((\gamma,\xi)=(0,0),(-8,\pi)\), where the first nonzero derivative has order four.

Thus for $\gamma \in \R \backslash \{-8,0\}$, we see from \(\phi''(0)=2\gamma\), \(\phi''(\pi)=-2(\gamma+8)\), and \(\phi''(\pi/2)=8\) that the second derivative changes sign somewhere. At such a zero the third derivative is nonzero, giving a cubic power by Theorem \ref{lemma-van der corput}. 
For $\gamma \in \{-8,0\}$, we use a fourth-derivative estimate. It is straightforward to verify that the two quadratic equations
\begin{equation*}
    -16c^2+2Bc+8=0 \,\,\mbox{and}\,\, 64c^2-2Bc-32=0
\end{equation*}
have no common root in $[-1,1]$. 

\subsection{Dimension two}
In dimension two, rank zero of Hess$\phi_{\gamma}$ forces \(s_1s_2=0\) from the off-diagonal entry. The diagonal entries then force \(A=0\) and \(s_1=s_2=0\).

When $\gamma \in [-16,0]$, the Hessian can have rank zero only when $\xi_j(j=1,2) \in \{-\pi,0,\pi\}$ for \(\gamma=0,-8,-16\). Recall that \(\gamma=0,-16\) have already been treated in Section \ref{sec-endpoint}. For \(\gamma=-8\), 
we assume $\xi_0 = (0,\pi)$ without loss of generality and expand the phase at this point to get
\begin{equation*}
    \begin{aligned}
        \phi_{-8}(\xi+\xi_0) 
        %& = (4-2\cos \xi_1 -2\cos(\xi_2+\pi))^2 -8(4-2\cos \xi_1- 2\cos(\xi_2+\pi)) \\
        & = (-4-2\cos \xi_1 +2\cos \xi_2)(4-2\cos\xi_1 +2\cos \xi_2) \\
        & = \tilde{c} + (\xi_1^2-\xi_2^2)^2 + \text{higher-order terms}.
    \end{aligned}
\end{equation*}
Thus, up to an additive constant and higher-order terms, \(\phi_{-8}\) is equivalent to \((\xi_1 ^2-\xi_2^2)^2\), and then by a linear change of variables to \(u^2v^2\). Then Theorem \ref{theorem-2dim NP uniform} gives \(t^{-1/2}\log t\).
Furthermore, every other point has Hessian rank at least one and admits the local \(t^{-1/2}\) estimate.

When $\gamma \in \mathcal E_2$, we have \(A\ne0\). If exactly one cosine vanishes, the Hessian is invertible. If both vanish, we expand the phase near \((\pi/2,\pi/2)\). In coordinates \((\xi_1,\xi_2)=(u+z,u-z)\),
\[
\varphi-\varphi_*=16u^2-2Au z^2-\frac{2A}{3}u^3+\text{higher-order terms}.
\]
Therefore, the Newton polyhedron is
\begin{equation*}
    \{(\lambda+1,2-2\lambda):\lambda \in [0,1]\}+\R^2_+.
\end{equation*}
Theorem \ref{theorem-2dim NP uniform} implies the local exponent \(3/4\).

If \(c_1c_2\ne0\) and the Hessian has rank one, set \(K=\sum s_j^2/c_j=-A/4\ne0\). Inspired by the null direction \(\tau_j=s_j/c_j\), we diagonalize the Hessian using 
\[
(\xi_1,\xi_2) = (\tau_1 y_1-\tau_2y_2,\tau_2y_1+\tau_1y_2),
\]
obtaining
\begin{equation*}
    \phi_{\gamma}(y) = \tilde{c} + a_1y_1^3 + a_2y_2^2 + \text{higher-order terms}.
\end{equation*}
Its cubic coefficient is
\[
a_1=\frac{4K}{3}\sum_{j=1}^2(c_j^{-3}+c_j^{-1}-2c_j).
\]
The sum on the right cannot vanish. Indeed, \(h(c)=c^{-3}+c^{-1}-2c\) is odd, strictly decreasing and nonnegative on \((0,1]\).
Cosines of the same sign cannot give zero, except at endpoints that force \(K=0\). For opposite signs, zero forces \(c_2=-c_1\), again implying \(K=0\). Note also that $a_2 \neq 0$ holds a priori since the Hessian has rank one. Again Theorem \ref{theorem-2dim NP uniform} gives exponent \(5/6\), stronger than the required \(3/4\).

\subsection{Dimension three}
It is sufficient to assume \(\gamma\in\mathcal E_3\) since the other cases have been handled by Proposition \ref{prop-interior} and Lemma \ref{lem-quartic}. By \eqref{eq-reflection}, we may take \(A=2\omega^2+\gamma>0\).

Although the integrals involve three variables, the explicit Hessian formula shows that exactly one, two, or three zero cosines give ranks three, two, or one, respectively. If all cosines are nonzero, the rank is at least two. Therefore, after splitting off the nondegenerate quadratic directions, at most two degenerate variables remain, so Theorem \ref{theorem-2dim NP uniform} is applicable. We consider the cases separately.

\underline{$\bullet$ All three vanishing cosines.} 
This case reduces to \((\pi/2,\pi/2,\pi/2)\). Set \(L=y_1+y_2+y_3\). Then
\[
\varphi-\varphi_*=4L^2-\frac A3\sum_jy_j^3+O(|y|^4).
\]
Taking \(y_1=z_1,y_2=z_2,y_3=z_3-z_1-z_2\), the weighted principal part is
\[
4z_3^2+A z_1z_2(z_1+z_2),\qquad\text{weights }(1/3,1/3,1/2).
\]
It is straightforward to verify that $\{z_1,z_2\}$ is an adapted coordinate system by Theorem \ref{theorem-jugde adapted coordinate} and hence
\[
M(\varphi-\varphi_*) \curlyeqprec M(4z_3^2+A z_1z_2(z_1+z_2))\curlyeqprec (-1/2,0) + M(z_1^2z_2+z_1z_2^2) \curlyeqprec (-7/6,0).
\]
%The binary cubic has three distinct real linear factors and uniform exponent \(-2/3\). Adding the square gives \(-7/6\).

\underline{$\bullet$ Exactly two zero cosines.}
After coordinate reflections and a permutation, consider
\[
\xi_*=\left(\frac{\pi}{2},\frac{\pi}{2},b\right),
\quad
c=\cos b\ne0,\quad s=\sin b,\quad a_*=\omega(\xi_*)^2=6-2c.
\]
We then introduce
\[
\xi_1=\frac{\pi}{2}+y_1+y_2,\qquad
\xi_2=\frac{\pi}{2}+y_1-y_2,\qquad
\xi_3=b+y_3,
\]
and define
\[
D(y)=\omega(\xi(y))^2-a_*
=4\sin y_1\cos y_2
 +2c(1-\cos y_3)+2s\sin y_3.
\]
The exact phase increment is
\begin{equation}
\label{eq-two-zero-exact}
F(y)=AD(y)+D(y)^2-4Ay_1-2Asy_3= F_2(y)+F_3(y)+O(|y|^4),
\end{equation}
where
\begin{align*}
F_2(y)
&=16y_1^2+16sy_1y_3+(4s^2+Ac)y_3^2,\\
F_3(y)
&=-\frac{2A}{3}y_1^3-2Ay_1y_2^2
  +8cy_1y_3^2
  +\left(4cs-\frac{As}{3}\right)y_3^3.
\end{align*}
In particular, \(y_2\) is the null direction of the Hessian.
The cubic term \(-2Ay_1y_2^2\) couples this direction to
the nondegenerate variables. Put
\[
K=4s^2+Ac,\qquad
B=\begin{pmatrix}
16&8s\\
8s&K
\end{pmatrix}.
\]
Since $B$ has full rank, we choose a rotation \(R\in SO(2)\) such that $R^{\mathsf T}BR
=\operatorname{diag}(\lambda_+,\lambda_-)$.

Make the change of variables
\[
\begin{pmatrix}y_1\\y_3\end{pmatrix}
=R\begin{pmatrix}z_1\\z_3\end{pmatrix},
\qquad z_2=y_2.
\]
The phase becomes
\[
\widetilde F(z)
=\lambda_+z_1^2+\lambda_-z_3^2
 +\widetilde F_3(z)+O(|z|^4).
\]

This rotation diagonalizes only the quadratic part.
To remove the higher-order couplings, apply the analytic
splitting lemma with \(z_2\) as a parameter. A further local
analytic change of the two nondegenerate variables, leaving
\(z_2\) unchanged, gives
\begin{equation*}
\widetilde F
=\varepsilon_1\eta_1^2+\varepsilon_3\eta_3^2+g(z_2),
\qquad \varepsilon_1,\varepsilon_3\in\{-1,1\}.
\end{equation*}
The Jacobian is nonzero and is absorbed into the amplitude.

Each nondegenerate quadratic variable contributes \(-1/2\)
to the uniform estimate. Hence the two quadratic
variables together contribute \(-1\). It remains to determine the first nonzero derivative
of the one-variable phase \(g\).

Since the Hessian block in \((y_1,y_3)\) is invertible,
the implicit function theorem gives a unique analytic branch
\[
y_1=h_1(\tau),\qquad y_3=h_3(\tau),\qquad \tau=y_2,
\]
satisfying
\[
\partial_{y_1}F(h_1(\tau),\tau,h_3(\tau))=0,\qquad
\partial_{y_3}F(h_1(\tau),\tau,h_3(\tau))=0,
\]
with \(h_1(0)=h_3(0)=0\).
Then
\[
g(\tau)=F(h_1(\tau),\tau,h_3(\tau)).
\]
The exact phase is even in \(y_2\). By uniqueness,
\(h_1,h_3\), and \(g\) are even functions.

Comparing the terms of order \(\tau^2\) in the stationary
equations gives
\[
B
\begin{pmatrix}h_1(\tau)\\h_3(\tau)\end{pmatrix}
=
\begin{pmatrix}A\\0\end{pmatrix}\tau^2+O(\tau^4).
\]
Therefore,
\[
h_1(\tau)=\frac{K}{16c}\tau^2+O(\tau^4),
\qquad
h_3(\tau)=-\frac{s}{2c}\tau^2+O(\tau^4).
\]
Substitution into the phase yields
\begin{equation}
\label{eq-two-zero-quartic}
g(\tau)=-\frac{AK}{16c}\tau^4+O(\tau^6).
\end{equation}
%For this calculation, note from
%\eqref{eq-two-zero-exact} that \(F(0,\tau,0)\equiv0\);
%there is no independent pure \(\tau^4\) term.
%The quartic coefficient arises from the quadratic form
%and the coupling \(-2Ay_1\tau^2\).

If \(K\ne0\), the coefficient in
\eqref{eq-two-zero-quartic} is nonzero. The one-dimensional
uniform estimate contributes \(-1/4\), and hence gives the local uniform exponent $-5/4$.

Suppose now that
\[
K=4s^2+Ac=0.
\]
Then \(s\ne0\), since otherwise \(K=Ac\ne0\).
To determine the next term, we use the exact phase
\eqref{eq-two-zero-exact}, rather than an unspecified
Taylor remainder.

Along the stationary branch, put
\[
\alpha(\tau)=A+2D(h_1(\tau),\tau,h_3(\tau)).
\]
The two stationary equations become
\[
\alpha(\tau)\cos h_1(\tau)\cos\tau=A,
\qquad
\alpha(\tau)\sin(b+h_3(\tau))=As.
\]
Since the quadratic coefficient of \(h_1\) vanishes,
write
\[
h_1(\tau)=u_4\tau^4+O(\tau^6).
\]
The first stationary equation gives
\[
\alpha(\tau)=A\sec\tau+O(\tau^8).
\]
Using the second equation and comparing fourth-order
coefficients in the identity \(\alpha=A+2D\), we obtain
\[
8u_4
=\frac{5A}{24}-\frac{2s^2}{3c}-\frac{s^4}{2c^3}
=-\frac{s^2(1+2c^2)}{2c^3},
\]
where the last equality uses \(Ac=-4s^2\).

Differentiation along the stationary branch eliminates
the derivatives of \(h_1\) and \(h_3\), so
\[
g'(\tau)
=\partial_{y_2}F(h_1(\tau),\tau,h_3(\tau))
=-4\alpha(\tau)\sin h_1(\tau)\sin\tau.
\]
Hence
\[
g'(\tau)=-4Au_4\tau^5+O(\tau^7),
\]
and integration gives
\begin{equation*}
g(\tau)
=-\frac{s^4(1+2c^2)}{6c^4}\tau^6+O(\tau^8).
\end{equation*}
The sextic coefficient is strictly negative.
Thus the remaining variable contributes \(-1/6\),
and the total local uniform exponent is $-7/6$.

%The uniformity in these two cases follows from the parameter-dependent splitting lemma and the one-dimensional van der Corput estimate: the relevant fourth or sixth derivative of the reduced phase remains bounded away from zero on a sufficiently small neighborhood under sufficiently small smooth perturbations.

%Consequently, points with exactly two zero cosines have
%local uniform exponent \(-5/4\) when \(K\ne0\), and
%\(-7/6\) when \(K=0\). Both cases satisfy the
%\(t^{-7/6}\) upper bound required for the global
%three-dimensional exterior-parameter estimate.

\underline{$\bullet$All three cosines nonzero.}
Let $\xi_*$ be a critical point of $\varphi(v_*,\xi)$, and set
\[
c_j=\cos\xi_{*,j}\ne0,\quad s_j=\sin\xi_{*,j},\quad s=(s_1,s_2,s_3)^{\mathsf T}\quad
q(\xi)=\omega(\xi)^2,\quad A_*=A(\xi_*)>0.
\]
Since
\[
D_\xi^2\varphi(v_*,\xi_*)
=2A_*\diag(c_1,c_2,c_3)+8ss^{\mathsf T},
\]
the Hessian has rank at least two. Rank three gives the
uniform local bound $O(|t|^{-3/2})$, so assume its rank is two.
The analytic splitting lemma gives
\[
\varphi(v_*,\xi)-\varphi(v_*,\xi_*)
=Q_2(u)+g(z),\qquad g(0)=g'(0)=g''(0)=0,
\]
where $Q_2$ is nondegenerate. We show that
$\operatorname{ord}_0 g\le6$.

Introduce the auxiliary phase
\begin{equation}\label{eq-auxiliary-phase}
\Phi(\xi,\alpha)
=v_*\cdot\xi+\alpha q(\xi)-\frac{(\alpha-\gamma)^2}{4}.
\end{equation}
Completing the square yields
\begin{equation}\label{eq-auxiliary-square}
\Phi(\xi,\alpha)
=\varphi(v_*,\xi)-\frac{(\alpha-A(\xi))^2}{4}.
\end{equation}
Thus the analytic coordinate $\rho=(\alpha-A(\xi))/2$
adds only a nondegenerate square and preserves the
residual degenerate germ. On the other hand,
$D_\xi^2\Phi=2\alpha\diag(\cos\xi_j)$ is invertible near
$(\xi_*,A_*)$, allowing all three $\xi$-variables to be
eliminated with $\alpha$ as a parameter.

Put $b_j=-v_{*,j}/2=A_*s_j$ and
$\varepsilon_j=\sgn(c_j)$. The stationary branch
$\xi=\Xi(\alpha)$ satisfies
\[
\sin\Xi_j(\alpha)=\frac{b_j}{\alpha},\qquad
\cos\Xi_j(\alpha)
=\varepsilon_j\frac{\sqrt{\alpha^2-b_j^2}}{\alpha},
\qquad \Xi(A_*)=\xi_*.
\]
Here $A_*>|b_j|$, so these expressions are analytic.
For $h(\alpha)=\Phi(\Xi(\alpha),\alpha)$, differentiation
along the stationary branch gives
\begin{equation}\label{eq-auxiliary-reduced-derivative}
h'(\alpha)
=q(\Xi(\alpha))-\frac{\alpha-\gamma}{2}
=-\frac{\mathcal H(\alpha)}{2\alpha},
\end{equation}
where
\begin{equation}\label{eq-auxiliary-H}
\mathcal H(\alpha)
=\alpha^2-(12+\gamma)\alpha
 +4\sum_{j=1}^3\varepsilon_j\sqrt{\alpha^2-b_j^2}.
\end{equation}
The Hessian of $\Phi$ has rank three by
\eqref{eq-auxiliary-square}. Splitting off its invertible
$\xi$-block therefore gives $h'(A_*)=h''(A_*)=0$, and hence
$\mathcal H(A_*)=\mathcal H'(A_*)=0$.

We claim that the multiplicity of this zero is at most five.
Write $B_j=b_j^2$ and
\[
T_\ell(\alpha)
=\sum_j\varepsilon_jB_j
       (\alpha^2-B_j)^{-5/2-\ell}.
\]
Direct differentiation gives
\[
\begin{aligned}
\mathcal H''(\alpha)
 &=2-4\sum_j\varepsilon_jB_j(\alpha^2-B_j)^{-3/2},\\
\mathcal H'''(\alpha)&=12\alpha T_0(\alpha),\\
\mathcal H^{(4)}(\alpha)&=12T_0(\alpha)-60\alpha^2T_1(\alpha),\\
\mathcal H^{(5)}(\alpha)&=-180\alpha T_1(\alpha)
                         +420\alpha^3T_2(\alpha).
\end{aligned}
\]
If the multiplicity were at least six, these identities
would imply $T_0(A_*)=T_1(A_*)=T_2(A_*)=0$.
Group equal positive values of $B_j$, and define
\[
x_B=(A_*^2-B)^{-1},\qquad
W_B=\Bigl(\sum_{j:B_j=B}\varepsilon_j\Bigr)
       B(A_*^2-B)^{-5/2}.
\]
Then $\sum_B W_Bx_B^\ell=0$ for $\ell=0,1,2$.
There are at most three distinct nodes $x_B$, so the
Vandermonde matrix has full column rank and all $W_B$ vanish.
Consequently,
\[
\mathcal H''(A_*)=2-4\sum_B W_B/x_B=2,
\]
a contradiction. By \eqref{eq-auxiliary-reduced-derivative},
\[
h(\alpha)-h(A_*)
=c_k(\alpha-A_*)^k+O((\alpha-A_*)^{k+1}),
\qquad c_k\ne0,\quad 3\le k\le6.
\]

The same order occurs in $g$. Indeed, the two splittings
of $\Phi$ identify its gradient local algebra with both
$\R\{z\}/(g'(z))$ and $\R\{\alpha-A_*\}/(h'(\alpha))$:
analytic coordinate changes preserve this algebra, and
nondegenerate quadratic variables disappear in the quotient.
The latter algebra has dimension $k-1$, so
$\operatorname{ord}_0g=k$.

Finally, apply the splitting lemma with the perturbation
as a parameter. The two quadratic directions remain
nondegenerate, and the $k$th derivative of the reduced phase remains bounded away from zero for sufficiently
small analytic perturbations. Uniform stationary phase in the quadratic variables, followed by the
van der Corput estimate, gives
\[
M\bigl(\varphi(v_*,\cdot),\xi_*\bigr)
\curlyeqprec(-1-1/k,0)
\curlyeqprec(-7/6,0).
\]

\begin{remark}
The exponent \(-7/6\) used in the preceding local estimates
need not be optimal at every critical point. 
The preceding analysis shows that the reduced phase may have order four, giving \(O(|t|^{-5/4})\) when exactly two cosines vanish. If all three cosines are nonzero and the Hessian has rank two, the reduced phase has order \(3\le k\le6\), giving \(O(|t|^{-1-1/k})\). Nondegenerate critical points give \(O(|t|^{-3/2})\).

Nevertheless, the common exponent is determined by the slowest local contributions. For every fixed \(\gamma\in(-\infty,-24)\cup(0,\infty)\), at \(\xi_*=(\pi/2,\pi/2,\pi/2)\), the sharp local exponent is $-7/6$.
\end{remark}

\section{Sharpness}\label{sec-sharp}

Throughout this section, \(d\ge2\), and we only consider
\(t\to+\infty\).
To prove sharpness of a uniform decay estimate, it often
suffices to obtain a nonzero leading term after a fixed
smooth localization near a critical point.
%Indeed, smooth frequency localization is bounded convolution in space, so a localized lower bound transfers to the spatial supremum of the full kernel. This does not require a pointwise comparison between the full and localized kernels at the same position.
Here we also allow time-dependent multipliers whose
supports shrink towards a fixed critical point.
The following lemma provides the required uniform
localization estimate and a lower-bound criterion.
%Its proof uses the summability and small-phase mechanisms of \cite[Section 5]{YZ}, formulated here for local multipliers.
For a multiplier \(a_t\), write
\[
G_{a_t}(x,t)
=(2\pi)^{-d}\int_{\mathbb T^d}
e^{i[x\cdot\xi+t\phi_\gamma(\xi)]}
a_t(\xi)\,d\xi.
\]

\begin{lemma}
%[Uniform localization and a small-phase lower bound]
\label{lem-multiplier}
\label{lem-box}
%The following statements hold.
Let \(L_t\) be invertible real \(d\times d\) matrices
with \(\|L_t\|\le C\) and $v_*=-\nabla\phi_\gamma(\xi_*)$.
\begin{enumerate}
\renewcommand{\labelenumi}{(\arabic{enumi})}

\item

Let \(b_t\in C_c^\infty(\mathbb R^d)\) have support
in a fixed compact set and uniformly bounded derivatives
of every fixed order.
Suppose that
\[
a_t(\xi_*+y)=b_t(L_t^{-1}y)
\]
is supported inside one torus coordinate chart and is
extended by zero outside that chart.
Then its Fourier coefficients \(k_t\) satisfy
\[
\sup_t\|k_t\|_{\ell^1(\mathbb Z^d)}<\infty.
\]
Consequently,
\begin{equation}
\label{eq-multbound}
\sup_{x\in\mathbb Z^d}|G_{a_t}(x,t)|
\le C\sup_{x\in\mathbb Z^d}|G_{d,\gamma}(x,t)|,
\end{equation}
with \(C\) independent of \(t\).
In particular, this conclusion applies to a fixed
smooth cutoff \(a_t=\psi\): take \(L_t=I_d\) and
\(b_t(y)=\psi(\xi_*+y)\).

\item
Let \(L_t\) satisfy \(L_t\to0\). Suppose that, uniformly on a fixed neighborhood of the origin,
\[
t\left\{
\phi_\gamma(\xi_*+L_tz)-\phi_\gamma(\xi_*)
+v_*\cdot L_tz
\right\}
\longrightarrow P(z),
\qquad P(0)=0,
\]
where \(P\) is continuous. Then, for all sufficiently large \(t\),
\[
\sup_{x\in\mathbb Z^d}|G_{d,\gamma}(x,t)|
\gtrsim |\det L_t|.
\]
\end{enumerate}
\end{lemma}

\begin{proof}
(1) Use the Fourier convention
\[
k_t(n)=(2\pi)^{-d}
\int_{\mathbb T^d}e^{in\cdot\xi}a_t(\xi)\,d\xi,
\qquad
a_t(\xi)=\sum_{n\in\mathbb Z^d}k_t(n)e^{-in\cdot\xi}.
\]
Changing variables \(y=L_tz\) and integrating by parts
give, for every fixed \(N\),
\[
|k_t(n)|
\le C_N|\det L_t|
       (1+|L_t^{\mathsf T}n|)^{-N}.
\]
For \(u\in n+[-1/2,1/2]^d\),
\[
|L_t^{\mathsf T}(u-n)|\le C\sqrt d/2.
\]
Thus the weights with arguments \(L_t^{\mathsf T}n\)
and \(L_t^{\mathsf T}u\) are uniformly comparable.
Integrating over each unit cube and summing yield,
for \(N>d\),
\begin{align*}
\sum_{n\in\mathbb Z^d}|k_t(n)|
&\le C_N|\det L_t|
 \int_{\mathbb R^d}(1+|L_t^{\mathsf T}u|)^{-N}\,du\\
&=C_N\int_{\mathbb R^d}(1+|z|)^{-N}\,dz
<\infty.
\end{align*}
The constant is independent of \(t\).
Since
\[
G_{a_t}(x,t)
=\sum_{n\in\mathbb Z^d}k_t(n)G_{d,\gamma}(x-n,t),
\]
the convolution inequality proves
\eqref{eq-multbound}.

\medskip
(2) Choose \(x_t\in\mathbb Z^d\) by taking a nearest integer
to each coordinate of \(tv_*\).
With
\[
\delta_t=x_t-tv_*,
\]
we then have \(|\delta_t|\le\sqrt d/2\).
Choose a fixed function
\[
b\in C_c^\infty(\mathbb R^d),\qquad
b\ge0,\qquad b\not\equiv0,
\]
supported sufficiently close to the origin that
\(|P(z)|<\pi/4\) on its support.
Set
\begin{equation*}
a_t(\xi_*+y)=b(L_t^{-1}y),
\end{equation*}
extended by zero outside the chart.
For all sufficiently large \(t\), this multiplier
satisfies the assumptions of (1).

Put
\[
F(y)=\phi_\gamma(\xi_*+y)-\phi_\gamma(\xi_*)
     +v_*\cdot y,
\qquad
\Theta_t=x_t\cdot\xi_*+t\phi_\gamma(\xi_*).
\]
The phase identity
\[
x_t\cdot(\xi_*+y)+t\phi_\gamma(\xi_*+y)
=\Theta_t+tF(y)+\delta_t\cdot y
\]
and the substitution \(y=L_tz\) give
\begin{align*}
\frac{e^{-i\Theta_t}G_{a_t}(x_t,t)}{|\det L_t|}
&=(2\pi)^{-d}
 \int e^{i[tF(L_tz)+\delta_t\cdot L_tz]}b(z)\,dz\\
&\longrightarrow
 (2\pi)^{-d}\int e^{iP(z)}b(z)\,dz.
\end{align*}
The rounding term tends to zero uniformly on the fixed
support of \(b\), because \(\delta_t\) is bounded and
\(L_t\to0\).
The limiting integral has strictly positive real part:
\[
\operatorname{Re}
\left((2\pi)^{-d}\int e^{iP(z)}b(z)\,dz\right)
\ge (2\pi)^{-d}2^{-1/2}\int b(z)\,dz>0.
\]
Hence
\[
|G_{a_t}(x_t,t)|\gtrsim|\det L_t|.
\]
Finally, (1) implies the full lower bound.
%\[
%|G_{a_t}(x_t,t)|
%\le C\sup_{x\in\mathbb Z^d}|G_{d,\gamma}(x,t)|.
%\]
%There is no restriction that \(x_t\) be independent of time: the localization estimate is uniform over all lattice points for each \(t\).
\end{proof}

\subsection{The interior interval}
\label{sec-shelllower}

For \(\gamma\in\mathcal I_d\), put
\[
\lambda=-\gamma/2\in(0,4d).
\]
Choose \(\xi_*\) with all coordinates equal and in
\((0,\pi)\), such that \(\omega(\xi_*)^2=\lambda\).
Then
\[
\nabla(\omega^2)(\xi_*)\ne0.
\]
Consequently, \(s=\omega^2-\lambda\), together with
\(d-1\) additional coordinates, defines a local
coordinate system. Choose a sufficiently small neighborhood \(U\) of
\(\xi_*\) on which \(\nabla(\omega^2)\ne0\), and a
fixed function
\[
\psi\in C_c^\infty(U),\qquad
\psi\ge0,\qquad \psi(\xi_*)>0.
\]
In Lemma~\ref{lem-multiplier}(1), take
\[
a_t(\xi)=\psi(\xi),\qquad
L_t=I_d,\qquad
b_t(y)=\psi(\xi_*+y).
\]
Only part (1) is needed in this application.

Since
\[
\phi_\gamma(\xi)
=(\omega(\xi)^2-\lambda)^2-\frac{\gamma^2}{4},
\]
the coarea formula at the lattice point \(x=0\) gives
\[
G_\psi(0,t)
=(2\pi)^{-d}e^{-it\gamma^2/4}
\int_{\mathbb R}e^{its^2}A(s)\,ds,
\]
where
\[
A(s)=
\int_{\omega^2=\lambda+s}
\frac{\psi}{|\nabla(\omega^2)|}\,d\sigma.
\]
The function \(A\) is smooth and compactly supported.
Moreover,
\[
A(0)=
\int_{\omega^2=\lambda}
\frac{\psi}{|\nabla(\omega^2)|}\,d\sigma>0:
\]
the level set is regular, and \(\psi\) is positive on
a relatively open neighborhood of \(\xi_*\) in that
level set. One-dimensional stationary phase yields
\[
G_\psi(0,t)
=(2\pi)^{-d}e^{-it\gamma^2/4}
\left(e^{i\pi/4}\sqrt{\pi}\,A(0)t^{-1/2}
+O(t^{-3/2})\right).
\]
Lemma~\ref{lem-multiplier}(1) therefore gives
\[
\sup_{x\in\mathbb Z^d}|G_{d,\gamma}(x,t)|
\gtrsim t^{-1/2}.
\]
This proves sharpness whenever the corresponding upper
bound is \(O(t^{-1/2})\).
The additional logarithm at \(d=2,\gamma=-8\) is
treated below.

\subsection{Power bounds: endpoints and exterior parameters}
\label{sec-powerlower}

At \(\gamma=0\), take
\[
\xi_*=0,\qquad v_*=0,\qquad L_t=t^{-1/4}I_d.
\]
Since $\phi_0(\xi)=|\xi|^4+O(|\xi|^6)$, the limiting phase is
\[
t\phi_0(L_tz)\longrightarrow |z|^4.
\]
Explicitly choose
\[
a_t(\xi)=b(t^{1/4}\xi),
\]
where \(b\) is fixed, smooth, nonnegative, nonzero,
and supported in a sufficiently small box.
Lemma~\ref{lem-multiplier}(2) gives
\[
\sup_{x\in\mathbb Z^d}|G_{d,0}(x,t)|
\gtrsim t^{-d/4}.
\]
At the other endpoint \(\gamma=-8d\), we simply use the reflected cutoff $a_t(\xi)=b(t^{1/4}(\pi\mathbf1-\xi))$.

%in a torus chart about \(\pi\mathbf1\). The reflection identity gives the same lower bound.

For exterior parameters, let
\[
\xi_*=(\pi/2,\ldots,\pi/2),\qquad
A_*=4d+\gamma\ne0,\qquad
v_*=-2A_*\mathbf1.
\]
Write
\[
y=z+re,\qquad
e=d^{-1/2}\mathbf1,\qquad z\perp e.
\]
Taylor expansion gives
\begin{align}
\label{eq-cornerexpand}
F(y)
&:=\phi_\gamma(\xi_*+y)-\phi_\gamma(\xi_*)
   +v_*\cdot y\notag\\
&=4dr^2-\frac{A_*}{3}\sum_{j=1}^dz_j^3
 +\text{cubic terms containing }r+O(|y|^4).
\end{align}

For \(d=3\) or \(d\ge5\), use normal scale \(t^{-1/2}\)
and tangential scale \(t^{-1/3}\).
With
\[
P_\parallel=ee^{\mathsf T},\qquad
P_\perp=I_d-P_\parallel,
\]
set
\[
L_t=t^{-1/3}P_\perp+t^{-1/2}P_\parallel
\]
and explicitly choose
\[
a_t(\xi_*+y)
=b\!\left(t^{1/3}P_\perp y+t^{1/2}P_\parallel y\right).
\]
Here \(b\) is fixed, nonnegative, nonzero, and supported
sufficiently close to the origin.
For \(Z=z+re\), the limiting phase is
\[
tF(L_tZ)
\longrightarrow
4dr^2-\frac{A_*}{3}\sum_{j=1}^dz_j^3.
\]
Indeed, the undisplayed cubic terms contain at least
one normal variable and vanish after rescaling, as
does the fourth-order remainder.
%Since
%\[
%\det L_t=t^{-1/2-(d-1)/3},
%\]
Lemma~\ref{lem-multiplier}(2) then yields
\[
\sup_{x\in\mathbb Z^d}|G_{d,\gamma}(x,t)|
\gtrsim t^{-1/2-(d-1)/3}
=t^{-(2d+1)/6}.
\]
%No full asymptotic expansion on a fixed neighborhood for the tangential cubic is required.

For \(d=2\), write $y=(u+z,u-z)$ and use the scales
\[
u=t^{-1/2}U,\qquad z=t^{-1/4}Z.
\]
The limiting phase is
\[
tF(t^{-1/2}U+t^{-1/4}Z,\,
     t^{-1/2}U-t^{-1/4}Z)
\longrightarrow 16U^2-2A_*UZ^2.
\]
Take
\[
a_t(\xi_*+(u+z,u-z))
=b(t^{1/2}u,t^{1/4}z),
\]
with
\[
L_t=
\begin{pmatrix}1&1\\1&-1\end{pmatrix}
\operatorname{diag}(t^{-1/2},t^{-1/4}).
\]
Again choosing \(b\) as in part (2), we obtain
\[
\sup_{x\in\mathbb Z^2}|G_{2,\gamma}(x,t)|
\gtrsim t^{-3/4}.
\]

\subsection{Logarithmic bound: dimension two}
\label{sec-log-two}

Let \(d=2\) and \(\gamma=-8\).
Near the mixed corner \(\xi_*=(0,\pi)\), write
\[
\xi=(\theta_1,\pi+\theta_2).
\]
The local coordinates
\[
u=2\sin(\theta_1/2),\qquad
v=2\sin(\theta_2/2)
\]
satisfy the exact identity
\[
\omega(\xi)^2-4=u^2-v^2.
\]
Set \(U=u+v\) and \(V=u-v\). Then
\[
\phi_{-8}(\xi)
=(\omega(\xi)^2-4)^2-16
=U^2V^2-16.
\]
Write the resulting fixed local coordinate map as
\[
\xi=\kappa(U,V),\qquad
J_\kappa(U,V)=|\det D\kappa(U,V)|>0.
\]
Choose a fixed nonnegative function
\(b\in C_c^\infty(\mathbb R^2)\), supported sufficiently
close to the origin, with \(b(0,0)>0\), and define
\[
a_t(\kappa(U,V))
=\frac{b(U,V)}{J_\kappa(U,V)}.
\]
After extension by zero, this is a time-independent
smooth multiplier.
In Lemma~\ref{lem-multiplier}(1), take
\[
L_t=I_2,\qquad b_t(y)=a_t(\xi_*+y).
\]

At \(x=0\), the Jacobian cancels the denominator in
the multiplier, giving
\[
G_{a_t}(0,t)
=(2\pi)^{-2}e^{-16it}
\int_{\mathbb R^2}e^{itU^2V^2}b(U,V)\,dU\,dV.
\]
The monomial asymptotic is
\[
\int e^{itU^2V^2}b(U,V)\,dU\,dV
=c\,b(0,0)t^{-1/2}\log t
+o(t^{-1/2}\log t),
\qquad c\ne0.
\]
It follows by stationary phase in \(V\), and then by
integration over \(t^{-1/2}\lesssim|U|\lesssim1\).
%it is also the two-dimensional vertex asymptotic in \cite[Theorem 2.8]{C24}.
Thus Lemma~\ref{lem-multiplier}(1) gives
\[
\sup_{x\in\mathbb Z^2}|G_{2,-8}(x,t)|
\gtrsim t^{-1/2}\log t.
\]

\subsection{Logarithmic bound: dimension four}
\label{sec-log}

We begin with the asymptotic analysis for a specific phase function. A similar estimate was proved in \cite{K94}. For a straightforward proof, see Appendix \ref{ssec-xyz}. 
%The small-phase lower bound \(t^{-3/2}\) alone does not recover the necessary logarithm. We use a time-dependent complex multiplier whose Fourier coefficients remain uniformly summable. This multiplier tests the original kernel through Lemma~\ref{lem-multiplier}(1).

\begin{lemma}
\label{lem-xyz}
Let \(\eta\in C_c^\infty(\mathbb R)\) be real and even,
with \(\eta(0)>0\). Define
\[
H(\Lambda)
=\int_{\mathbb R^3}
e^{i\Lambda uvw}\eta(u)\eta(v)\eta(w)\,du\,dv\,dw.
\]
As \(|\Lambda|\to\infty\),
\[
H(\Lambda)
=\frac{4\pi\eta(0)^3}{|\Lambda|}\log|\Lambda|
+O(|\Lambda|^{-1}).
\]
\end{lemma}

Now let \(d=4\) and \(\gamma\in\mathcal E_4\).
Use the notation of \eqref{eq-cornerexpand}, so
\[
\xi_*=(\pi/2,\pi/2,\pi/2,\pi/2),\qquad
A_*=16+\gamma\ne0.
\]
Write \(z\in e^\perp\) as
\[
z=(a,b,c,-a-b-c).
\]
The identity
\[
\sum_{j=1}^4z_j^3=-3(a+b)(a+c)(b+c)
\]
suggests the invertible linear coordinates
\[
u=a+b,\qquad v=a+c,\qquad w=b+c.
\]
In these coordinates, the phase increment is exactly
\begin{equation}
\label{eq-fourphase}
F(r,u,v,w)=16r^2+A_*uvw+R(r,u,v,w).
\end{equation}

Assign weight \(1/2\) to \(r\), and weight \(1/4\) to
each of \(u,v,w\). Indeed every Taylor monomial of \(R\) has weight at least one.
%Indeed, the full quadratic part has already been displayed; every remaining cubic term contains \(r\), and every term of ordinary degree at least four has weight at least one.
Since \(R\) is analytic, on a fixed sufficiently small box the functions
\begin{equation}
\label{eq-Rbounded}
tR(t^{-1/2}\rho,t^{-1/4}U,t^{-1/4}V,t^{-1/4}W)
\end{equation}
have uniformly bounded derivatives of every fixed order.

Let \(x_t\in\mathbb Z^4\) be obtained by taking a nearest integer to each coordinate of \(tv_*\), and put
\(\delta_t=x_t-tv_*\).
Choose nonnegative even functions
\(\chi,\eta\in C_c^\infty(\mathbb R)\), supported
sufficiently close to zero, with \(\eta(0)>0\) and
\[
C_\chi
=\int_{\mathbb R}e^{16i\rho^2}\chi(\rho)\,d\rho\ne0.
\]
%For example, a nonzero \(\chi\ge0\) supported where \(16\rho^2<\pi/4\) has this property.

In the above linear coordinates, define
\begin{align}
\label{eq-corrector}
a_t(\xi_*+y)
={}&\chi(t^{1/2}r)
\eta(t^{1/4}u)\eta(t^{1/4}v)\eta(t^{1/4}w)
\notag\\
&\times
\exp\{-itR(r,u,v,w)-i\delta_t\cdot y\}.
\end{align}
This multiplier cancels both the remainder and the
rounding term in the localized integral.

To verify the assumptions of
Lemma~\ref{lem-multiplier}(1), let \(T_0\) denote the
fixed invertible linear map from \((r,u,v,w)\) to \(y\),
and take
\[
L_t=T_0\operatorname{diag}
(t^{-1/2},t^{-1/4},t^{-1/4},t^{-1/4}).
\]
For \(Z=(\rho,U,V,W)\), define
\[
b_t(Z)=a_t(\xi_*+L_tZ).
\]
Explicitly,
\begin{align*}
b_t(Z)
={}&\chi(\rho)\eta(U)\eta(V)\eta(W)\\
&\times\exp\bigl\{
-itR(t^{-1/2}\rho,t^{-1/4}U,
     t^{-1/4}V,t^{-1/4}W)
-i\delta_t\cdot L_tZ
\bigr\}.
\end{align*}
Its support is fixed.
Equation \eqref{eq-Rbounded}, the boundedness of
\(\delta_t\), and the boundedness of \(L_t\) imply
uniform bounds for every fixed derivative order.
Thus
\[
\sup_t\|k_t\|_{\ell^1}<\infty
\]
for the Fourier coefficients of this multiplier.

Evaluate \(G_{a_t}\) at \(x=x_t\), and remove the
constant phase
\[
\Theta_t=x_t\cdot\xi_*+t\phi_\gamma(\xi_*).
\]
By \eqref{eq-fourphase} and \eqref{eq-corrector}, the
remaining phase is exactly \(16tr^2+tA_*uvw\).
%With the fixed positive constant
%\[
%C_0=(2\pi)^{-4}|\det T_0|,
%\]
Then we obtain
\begin{align*}
e^{-i\Theta_t}G_{a_t}(x_t,t)
={}&C_0
\left(\int e^{16itr^2}\chi(t^{1/2}r)\,dr\right)\\
&\times
\left(\int e^{itA_*uvw}
\eta(t^{1/4}u)\eta(t^{1/4}v)\eta(t^{1/4}w)
\,du\,dv\,dw\right)\\
={}&C_0 C_\chi\,t^{-5/4}H(A_*t^{1/4}).
\end{align*}
Lemma~\ref{lem-xyz} therefore gives
\[
|G_{a_t}(x_t,t)|
=c_{\gamma,\chi,\eta}\,t^{-3/2}\log t
+O(t^{-3/2}),
\qquad c_{\gamma,\chi,\eta}>0.
\]
Applying Lemma~\ref{lem-multiplier}(1) yields
\begin{equation*}
\sup_{x\in\mathbb Z^4}|G_{4,\gamma}(x,t)|
\gtrsim t^{-3/2}\log t,
\qquad \gamma\in\mathcal E_4.
\end{equation*}
%Together with the upper bounds proved earlier, these lower bounds establish sharpness in all the cases treated in this section. Negative times follow from
%\[
%G_{d,\gamma}(x,-t)=\overline{G_{d,\gamma}(-x,t)}.
%\]
%Finally, the convolution operator norm from
%\(\ell^1(\mathbb Z^d)\) to \(\ell^\infty(\mathbb Z^d)\) equals \(\|G_{d,\gamma}(\cdot,t)\|_{\ell^\infty}\).

\section{Appendix}\label{sec-app}

\subsection{The discrete setting}\label{ssec-disc set}
Recall that $\mathbb{Z}^d$ denotes the standard integer lattice graph in $\mathbb{R}^d$. A function $h:\Z^d \rightarrow \mathbb C$ is said to be rapidly decreasing or to belong to the Schwartz class if $\sup_{x \in \Z^d} |x|^k|h(x)| \leq C_k$ for $k = 0,1,2,\cdots$. For such a function, we define its Fourier transform by
\begin{equation*}
    \mathcal{F}(h)(\xi)=\Hat{h}(\xi) := \sum_{x\in \mathbb{Z}^d} e^{-i\xi \cdot x}h(x),~~\xi \in \mathbb{T}^d.
\end{equation*}

The inverse transform for a smooth function $h: \mathbb{T}^d \rightarrow \mathbb C$ is defined by
\begin{equation*}
    \mathcal{F}^{-1}(h)(x)=\Check{h}(x) := \frac{1}{(2\pi)^d} \int_{\mathbb{T}^d} e^{i \xi \cdot x} h(\xi) d\xi,~~x\in \mathbb{Z}^d.
\end{equation*}
We remark here that the Fourier transform and its inverse can be defined for more general functions. The convolution property and the Plancherel identity can be proved similarly. See e.g. \cite{W11} for more facts about discrete Fourier analysis. Applying the Fourier transform to both sides of (\ref{equ-DFS}) (with $F \equiv 0$), we get
\begin{equation*}
    \left\{
        \begin{aligned}
            & \partial_t \Hat{u}(\xi,t) = i(\omega^4+\gamma \omega^2) \Hat{u}(\xi,t),\\
            & \Hat{u}(\xi,0) = \Hat{f}(\xi).
        \end{aligned}
    \right.
\end{equation*}
For any fixed $\xi$, the solution to this ordinary differential equation is 
\begin{equation*}
    \Hat{u}(\xi,t) = e^{i(\omega^4+\gamma \omega^2)t} \hat{f}(\xi).
\end{equation*}
Therefore we have
\begin{equation}\label{equ-semigroup solution}
    u(x,t)=\frac{1}{(2\pi)^d} \int_{\mathbb{T}^d} e^{i \xi \cdot x+i(\omega^4+\gamma \omega^2)t} \hat{f}(\xi) d\xi.
\end{equation}

\subsection{Uniform estimates for oscillatory integrals}\label{ssec-general OI}

The following notions are taken from \cite{K83}. In the sequel, let $B_{\R^d}(\xi,r)$ (resp. $B_{\mathbb C^d}(\xi,r)$) be the usual open ball in $\R^d$ (resp. $\mathbb C^d$) with center $\xi$ and radius $r$, while $\overline{B}_{\R^d}(\xi,r)$ (resp. $\overline{B}_{\mathbb C^d}(\xi,r)$) denotes its closure.

\begin{definition}
    For $r,s>0$, let $H_r(s)$ denote the class of complex-valued functions $P$ that are holomorphic on $B_{\mathbb C^d}(0,r)$, continuous on $\overline{B}_{\mathbb C^d}(0,r)$, and satisfy $|P(w)|<s$, $\forall \, w\in \overline{B}_{\mathbb C^d}(0,r)$.
\end{definition}
    
For any given $r,s>0$, one may verify that real polynomials belong to $\mathcal{H}_{r}(s)$ if their coefficients are sufficiently small.

\begin{definition}\label{def-uniform estimate}
    Suppose that $h:\mathbb{R}^d \rightarrow \mathbb{R}$ is real-analytic at 0 and $(\beta,p)\in\R\times\mathbb N$. We write 
    \begin{equation*}
        M(h) \curlyeqprec (\beta,p)
    \end{equation*} 
    if, for every sufficiently small $r > 0$, there exist $\epsilon>0$, $C>0$ and a neighborhood $A\subset B_{\R^d}(0,r)$ of the origin such that
    \begin{equation*}
         \bigg| \int_{\mathbb{R}^d}e^{it(h(x)+P(x))}\psi(x) dx \bigg| \leqslant C (1+|t|)^{\beta} \log^p(2+|t|) \|\psi\|_{C^N(A)}
    \end{equation*}
    holds for any $t \in \R$, $\psi \in C^{\infty}_0(A)$ and $P\in \mathcal{H}_r(\epsilon)$. Here $N \in \mathbb N$ only depends on $d$, and $\|\psi\|_{C^N(A)}:=\sup\{|\partial^{\gamma}\psi(x)|:x\in A,\gamma\in\mathbb N^d,|\gamma|\leq N\}.$
\end{definition}

By $M(h,a) \curlyeqprec (\beta,p)$ we mean that $h_a(x):=h(x+a)$ has a uniform estimate at 0 with exponent pair $(\beta,p)$, i.e. $M(h_a) \curlyeqprec (\beta,p)$.

One-dimensional uniform estimate follows from the van der Corput lemma, see e.g. \cite[Section 8]{S93}.

\begin{theorem}\label{lemma-van der corput}
    Suppose that $h$ is a real-valued smooth function, and $\psi$ is complex-valued and smooth. Suppose also that $|h^{(k)}(x)| \geqslant 1$ for some $k \geqslant 2$ and all $x \in (a,b)$. Then we can conclude that
    \begin{equation*}
        \bigg| \int_{a}^{b} e^{i\lambda h(x)} \psi(x)dx \bigg| \leqslant c_k \lambda^{-1/k} \left[|\psi(b)|+ \int_a^b \psi'(x)dx\right].
    \end{equation*}
    The constant $c_k = 5 \cdot 2^{k-1}-2$ is independent of $h$, $\psi$ and $\lambda$. 
\end{theorem}

\subsection{Newton polyhedra}\label{ssec-Newton polyhedra}
We summarize the Newton-polyhedron method developed in \cite{V76}. 

Without loss of generality, we can and will assume that $S:\R^d \rightarrow \R$ is an analytic function defined in a neighborhood of the origin with
\begin{equation*}
    S(0)=0\,\,\,\mbox{and}\,\,\,\nabla S(0)=0. 
\end{equation*}
$S$ can be expressed as a uniformly and absolutely convergent series, so we consider the associated Taylor expansion at the origin
\begin{equation*}
    S(x) = \sum_{m \in \mathbb{N}^d} s_{m} x^{m}.
\end{equation*}
The set
\begin{equation*}
    \mathcal{T}(S) := \{m \in \mathbb{N}^d :s_{m} \neq 0 \}
\end{equation*} 
is called the Taylor support of $S$, and we will always assume $\mathcal{T}(S) \neq \varnothing$. The Newton polyhedron of $S$, denoted by $\mathcal{N}(S)$, is the convex hull of the set
\begin{equation*}
    \bigcup_{m \in \mathcal{T}(S)} \left(m + \mathbb{R}^d_{+}\right),
\end{equation*}
where $\R^d_+=\{x\in \R^d:x_j\geq 0, \, 1 \leq j\leq d\}$. Given a compact face $\mathcal{P}$ of $\mathcal{N}(S)$, we write
\begin{equation*}
    S_{\mathcal{P}}(x) = \sum_{m \in \mathcal{P}}s_{m} x^{m}. 
\end{equation*}
Note that an edge or a vertex is also considered as a face. We say that $S$ is $\mathbb{R}$-nondegenerate if $\nabla S_{\mathcal{P}}(x)$ is nonvanishing on $(\mathbb{R}\backslash\{0\})^d$ for all compact faces $\mathcal{P}$ of $\mathcal{N}(S)$, that is,
\begin{equation*}
   \bigcup_{\mathcal{P}}\left( \bigcap_{j=1}^d \left\{x:{\partial_{j}S_{\mathcal{P}}}(x)=0\right\}\right)\subset\; \bigcup_{j=1}^d \{x:x_j=0\}.
\end{equation*}

The {Newton distance} $d_S$ of $\mathcal{N}(S)$ is defined by
\begin{equation*}
        d_S = \inf \{\rho>0:(\rho,\rho,\cdots,\rho) \in \mathcal{N}(S)\} ,
\end{equation*}
and the principal face $\pi(S)$ of $\mathcal{N}(S)$ is the face of minimal dimension containing the point 
$\boldsymbol{d_S}=(d_S,d_S,\cdots,d_S)$. We shall call the series
\begin{equation*}
    S_{pr}:=\sum_{m \in \pi(S)} s_m x^m
\end{equation*}
the principal part of $S$. If $\pi(S)$ is compact, $S_{pr}$ is a polynomial; otherwise, we consider $S_{pr}$ as a formal power series.

Following \cite{IKM05} and \cite{IM11-TAMS}, we also define the order of $S$ at a point $y$, denoted by ord$S(y)$, as the smallest integer $j$ such that $D^j S(y) \neq 0$, where $D^j S$ denotes the $j$th-order total derivative of $S$. If $S: \R^2 \rightarrow \R$, we let
\begin{equation*}
    m(S):= \max_{y \in \mathbb{S}^1} \mathrm{ord}S(y)
\end{equation*}
be the maximal order of $S$ along the unit circle in $\R^2$.

In his seminal work \cite{V76}, Varchenko proved that the leading term of the asymptotic expansion of an oscillatory integral can be determined from the Newton polyhedron of its phase function under certain conditions. See also \cite[Theorem 2.3]{G18}.

\begin{theorem}\label{theorem-newton polyhedra}
    Suppose $S$ is $\mathbb{R}$-nondegenerate and $\psi: \mathbb{R}^d \rightarrow \mathbb{R}$ is smooth and supported sufficiently close to the origin. Let $k_0\in\N$ be the greatest codimension over all faces of $\mathcal{N}(S)$ containing the point $\boldsymbol{d_S}$, i.e. $k_0=d-dim(\pi(S))$. Then
    \begin{equation*}
       \left|\int_{\R^d}e^{itS(x)}\psi(x)\,dx\right|\leqslant C (1+|t|)^{-\frac{1}{d_S}} \log^{k_0-1}(2+|t|).
    \end{equation*}
\end{theorem}

\subsection{Integrals in two variables}\label{ssec-2dim OI}

In this part we focus on integrals on $\R^2$. From the definitions above, we may find that the Newton distance depends on the chosen local coordinate system in which $S$ is expressed. Based on this, the height of $S$ is defined by 
\begin{equation*}
    ht_S :=\sup \{d_{S,x}\},
\end{equation*}
where the supremum is taken over all local analytic coordinate systems preserving 0 and $d_{S,x}$ is the Newton distance in coordinates $\{x\}$. A given coordinate system $\{x_0\}$ is said to be adapted to $S$ if $d_{S,x_{0}} = ht_S$. 

For a certain function, it is necessary to ask whether an adapted coordinate system exists. Varchenko \cite{V76} proved the existence of such a coordinate system for analytic functions (without multiple components) in the two-dimensional case and
showed that the analogous statement fails in higher dimensions. Also in dimension two, Ikromov and M\"{u}ller \cite{IM11-TAMS} extended Varchenko's result to the smooth setting. Therefore, we may and will assume that $S$ is expressed in an adapted coordinate system from now on. 

They also gave necessary and sufficient conditions for the adaptedness of a given system, see \cite[Corollary 2.3 and 4.3]{IM11-TAMS}.
\begin{theorem}\label{theorem-jugde adapted coordinate}
    The coordinates $\{x_0\}$ are adapted to $S$ if and only if one of the following conditions is satisfied:
    \begin{itemize}
        \item[(a)] $\pi(S)$ is a compact edge, and $m(S_{pr}) \leq d_{S,x_0}$.
        \item[(b)] $\pi(S)$ is a vertex.
        \item[(c)] $\pi(S)$ is an unbounded edge.
    \end{itemize}
\end{theorem}

%From Theorem \ref{theorem-jugde adapted coordinate}, we deduce that if a coordinate system is not adapted to $S$, then $\pi(S)$ is a compact edge. In $\R^2$, it means that $\pi(S)$ lies on a line $\kappa_1 x_1 + \kappa_2 x_2 =1$. Without loss of generality, we may assume $\kappa_2 \geq \kappa_1 \geq 0$. Then there exists a smooth function $\psi(x_1)$ of $x_1$ with $\psi(0)=0$ such that an adapted coordinate system $\{y_1,y_2\}$ is given by $y_1=x_1, y_2=x_2 - \psi(x_1)$. See e.g. \cite[Theorem 5.1]{IM11-TAMS}. Therefore, we may and will assume that $S$ is expressed in an adapted coordinate system from now on. 

We also need to determine Varchenko's exponent $\nu(S) \in \{0,1\}$ (see \cite[p. 1294]{IM11-JFAA}), which is also called the multiplicity of $S$ in \cite{K84}. If there exists an adapted coordinate system such that $\pi(S)$ is a vertex, and if $ht_S \geq 2$, then we let $\nu(S) = 1$; otherwise, we let $\nu(S) = 0$. 

\iffalse
However, this condition is not always easy to verify; a more accessible criterion is given in
\cite[Lemma 1.5]{IM11-JFAA}.

\begin{theorem}\label{theorem-judge nu}
    The following conditions on $S$ are equivalent:
    \begin{itemize}
        \item [(a)] There exists an adapted local coordinate system such that $\pi(S)$ is a vertex.
        \item [(b)] For any adapted local coordinate system $\{y\}$, either $\pi(S)$ is a vertex, or a compact edge and $m(S_{pr})=d_{S,y}$.
    \end{itemize}
\end{theorem}
\fi

%For oscillatory integrals in two variables, we do not need to verify the $\R$-nondegenerate condition for the phase as in Theorem \ref{theorem-newton polyhedra}. 
Under the assumption that $S$ is expressed in an adapted coordinate system, we obtain the stable decay estimate, see \cite[Theorem 1.1]{IM11-JFAA} and \cite[Theorem 2.1]{K84}. 

\begin{theorem}\label{theorem-2dim NP uniform}
    Suppose $S:\R^2 \rightarrow \R$ and $\psi: \mathbb{R}^2 \rightarrow \mathbb{R}$ is smooth and supported sufficiently close to the origin. Let $ht_S$ and $\nu(S)$ be defined as above. Then
    \begin{equation*}
       \left|\int_{\R^2}e^{itS(x)}\psi(x)\,dx\right|\leqslant C (1+|t|)^{-\frac{1}{ht_S}} \log^{\nu(S)}(2+|t|).
    \end{equation*}
    Furthermore, the stable estimate holds:
    \begin{equation*}
        M(S) \curlyeqprec (-1/{ht_S},\nu(S)).
    \end{equation*}
\end{theorem}

\iffalse
Furthermore, the following result on stability of oscillatory integrals is valid, see \cite[Theorem 2.1]{K84}.

\begin{theorem}\label{theorem-2dim NP uniform}
    Let $S$, $\psi$, $ht_S$ and $\nu(S)$  be as in Theorem \ref{theorem-2 dim newton polyhedra}. Moreover, if the inequality
    \begin{equation*}
        \left| \int_{\mathbb{R}^2} e^{itS(x)} \psi(x)dx \right| \leqslant C(1+|t|)^{\beta_S} \log^{p_S}(2+|t|)
    \end{equation*}
    holds with some exponent pair $(\beta_S,p_S)$,
    then $M(S) \curlyeqprec (\beta_S,p_S)$.
    In particular, $M(S) \curlyeqprec (-1/{ht_S},\nu(S))$.
\end{theorem}
\fi

When $d>2$, it is generally difficult to establish uniform estimates. However, at noncritical points and nondegenerate critical points of the phase, the following estimates hold.
\begin{lemma}\label{lemma-regular and nondegenerate estimate}
    Suppose that $h:\R^d \rightarrow \R$ is real analytic at some point $x_0$.
    \begin{itemize}
        \item [(a)] If $\nabla h(x_0) \neq 0$, then $M(h,x_0) \curlyeqprec (-N,0)$ for any $N \in \N$.
        \item [(b)] If $x_0$ is a nondegenerate critical point of $h$, then $M(h,x_0) \curlyeqprec (-d/2,0)$.
    \end{itemize}
\end{lemma}

\subsection{Proof of Lemma \ref{lem-xyz}}\label{ssec-xyz}
Use the convention
\[
\widehat\eta(\zeta)
=\int_{\mathbb R}e^{-i\zeta w}\eta(w)\,dw.
\]
For \(\Lambda>0\), integration in \(w\) and evenness give
\[
H(\Lambda)
=4\int_0^\infty\int_0^\infty
\eta(u)\eta(v)\widehat\eta(\Lambda uv)\,dv\,du.
\]
The region \(0<u<\Lambda^{-1}\) contributes
\(O(\Lambda^{-1})\).
For \(u\ge\Lambda^{-1}\), set \(z=\Lambda uv\).
The Schwartz decay of \(\widehat\eta\) gives
\[
\int_0^\infty\eta(v)\widehat\eta(\Lambda uv)\,dv
=
\frac{\eta(0)}{\Lambda u}
\int_0^\infty\widehat\eta(z)\,dz
+O((\Lambda u)^{-2}).
\]
The error, integrated against \(\eta(u)\) for
\(u\ge\Lambda^{-1}\), is \(O(\Lambda^{-1})\).
Fourier inversion and evenness imply
\[
\int_0^\infty\widehat\eta(z)\,dz=\pi\eta(0),
\]
while
\[
\int_{\Lambda^{-1}}^\infty\frac{\eta(u)}{u}\,du
=\eta(0)\log\Lambda+O(1).
\]
These identities prove the stated asymptotic.
Finally, replacing \(u\) by \(-u\) shows that
\(H(-\Lambda)=H(\Lambda)\).

%\section*{Acknowledgement}

\section*{AI Disclosure}

The author has used the LLM GPT-6 during the development of this work. It has been employed to classify possible cases and assist with calculations. All mathematical statements and proofs have been checked and rewritten by the author, who takes full responsibility for the mathematical content.

%\nocite{BCH4,Karp,Stein}
%\bibliographystyle{amsplain}
%\bibliography{references}

\printbibliography

\end{document}